\documentclass[11pt]{amsart}

\usepackage{amsmath}
\usepackage{amssymb}
\usepackage{enumitem}
\usepackage{derivative}
\usepackage{commath}
\usepackage{mathtools}
\usepackage{comment}
\usepackage{mathrsfs}
\usepackage{hyperref}
\usepackage{mathabx}
\usepackage{xcolor}
\usepackage{derivative}
\usepackage{inputenc}
\usepackage{graphicx}
\usepackage{ upgreek }

\theoremstyle{plain}
\newtheorem{theorem}{Theorem}
\newtheorem{lemma}[theorem]{Lemma}
\newtheorem{proposition}[theorem]{Proposition}
\newtheorem{corollary}[theorem]{Corollary}

\newtheorem{remark}[theorem]{Remark}
\numberwithin{theorem}{section}
\numberwithin{equation}{section}

\newcommand{\B}{\mathbb{B}}
\newcommand{\C}{\mathbb{C}}
\newcommand{\N}{\mathbb{N}}

\newcommand{\R}{\mathbb{R}}

\newcommand{\X}{\mathbb{X}}

\newcommand{\U}{\mathrm{U}}

\newcommand{\fL}{\mathfrak{L}}

\newcommand{\cT}{\mathcal{T}}
\newcommand{\cA}{\mathcal{A}}

\newcommand{\cS}{\mathcal{S}}
\newcommand{\cF}{\mathcal{F}}

\newcommand{\rU}{\mathrm{U}}

\newcommand{\vn}{\textbf{\textit{n}}}

\newcommand{\ip}[2]{\langle #1,#2 \rangle}

\newcommand{\bddfG}{L^\infty(G)}

\newcommand{\F}{\cF^2(\C^n)}

\newcommand{\bddf}{L^\infty(\C^n)}
\newcommand{\Ltwof}{L^2(\C^n)}

\newcommand{\cbuC}{\mathcal{C}_{b,u}(\C^n)}
\newcommand{\cbuCGinvar}{\mathcal{C}_{b,u}(\C^n)^G}

\newcommand{\Berg}{\cA^2(\mathbb{B}^n)}
\newcommand{\bergfock}{\cA^2(\mathbb{X}^n)}

\newcommand{\homf}{\bddf^G}
\newcommand{\bdd}{\fL(\cF^2)}
\newcommand{\tclass}{\cS^1(\cF^2)}

\newcommand{\toepalgG}{\mathfrak{T}^G}

\newcommand{\toep}{\cT_{L^\infty}}

\newcommand{\toepalg}{\mathfrak{T}(L^\infty)}
\newcommand{\toepalgrad}{\mathfrak{T}^{qrad}(L^\infty)}

\newcommand{\cbuop}{C_{b,u}(\cF^2)}

\newcommand{\cbuopG}{C_{b,u}(\cF,G)}
\newcommand{\cbuopGinvar}{C_{b,u}(\cF)^G}

\newcommand{\bddX}{\fL(\bergfock)}
\newcommand{\bddqradX}{\fL(\bergfock)^{qrad}}
\newcommand{\bddqradBerg}{\fL(\Berg)^{qrad}}
\newcommand{\Gast}{\ast_G}

\newcommand{\Hom}{\mathrm{Hom}(G,\piG)}

\newcommand{\Avg}{\psi\ast_G}
\newcommand{\avg}{\psi \ast_G}
\newcommand{\piG}{\pi_G}
\newcommand{\TpiG}{\tilde{\pi}_G}

\newcommand{\conv}[1]{\varphi \ast #1}

\newcommand{\haar}[1]{\ d\mu(#1)}

\newcommand{\actf}[1]{\tau_{#1}}
\newcommand{\actop}[1]{L_{#1}}
\newcommand{\actm}[1]{L_{#1}}

\newcommand{\QRad}[1]{Q\text{-}Rad(#1)}
\newcommand{\qrad}[1]{q\text{-}rad(#1)}
\newcommand{\intG}[1]{\int_{G} #1 \ d\mu(A)}

\newcommand{\Grota}{G_2}
\newcommand{\Gtrans}{G_1}

\begin{document}


\keywords{Toeplitz operators, Fock space, Bergman space, commutative $C^*$-algebras, quantum harmonic analysis, }
\thanks{M.~M. was supported in part by NSF grant DMS-2000236.}
\title[Density of Toeplitz operators in Toeplitz subalgebras]{Density of Toeplitz operators in rotation-invariant Toeplitz algebras}

\author{Vishwa Dewage \& Mishko Mitkovski}
\address{Department of Mathematics, Clemson University\\
South Carolina, SC 29634, USA\\
E-Mail Dewage:vdewage@clemson.edu
E-Mail Mitkovski: mmitkov@clemson.edu}
\subjclass[2000]{22D25, 30H20, 47B35, 47L80}
 
\maketitle
\begin{abstract} We use results and techniques from Werner's ``quantum harmonic analysis'' to show that $G$-invariant Toeplitz operators are norm dense in $G$-invariant Toeplitz algebras for all subgroups $G$ of the affine unitary group $U_n\ltimes \mathbb{C}^n$. Additionally, we prove that the quasi-radial Toeplitz operators are dense in the quasi-radial Toeplitz algebra over the Bergman space $\mathcal{A}^2(\mathbb{B}^n)$ and provide a constructive proof of SOT density of Toeplitz operators in the space of all bounded operators.
\end{abstract}


\section{Introduction}
Recently there has been a reemerging interest \cite{LS18,F19, FG23, FR23} in Werner's paper \cite{W84}  on ``quantum harmonic analysis''. At the heart of Werner's work was the quantum form of Wiener's Tauberian theorem. Many of his results were revisited and extended recently by Luef and Skrettingland \cite{LS18}. A particularly important observation was made by Fulsche \cite{F19} who showed that quantum harmonic analysis results and techniques can be applied in the study of Toeplitz algebras over the Fock space. More precisely, Fulsche proved that the algebra of bounded uniformly continuous operators on the Fock space, an important object of study in \cite{W84}, coincides with the Toeplitz algebra over the Fock space $\F$. In addition, as an application of the generalized Wiener's Tauberian theorem, he gave a different proof of the well-known important result of Xia \cite{X15} that the Toeplitz algebra is the norm closure of the set of all Toeplitz operators with bounded symbols. Lastly, as an application of the theory, Fulsche also proved the density of Toeplitz operators in commutative $G$-invariant Toeplitz algebras for fairly general additive subgroups $G$ of $\C^n$. However, as Fulsche himself noticed, his results exclude well-known commutative algebras such as the radial/quasi-radial Toeplitz algebras \cite{BV12, DO22, DOQ18, EM16, GKV03, GMV13, GV02, V10}, which are invariants of the rotation group and its subgroups. It was a general feeling, of at least some of the experts, that methods of quantum harmonic analysis are somehow tightly connected to $\C^n$ and its additive subgroups, and are therefore not applicable to algebras like the radial or the quasi-radial Toeplitz algebras. The main goal of this note is to show that this is not the case. We prove the density of Toeplitz operators with bounded uniformly continuous $G$-invariant symbols in $G$-invariant Toeplitz algebras for groups $G$ which are subgroups of $U_n\ltimes \C^n$. In particular, our results apply to radial or quasi-radial Toeplitz algebras. 

In \cite{B22}, Bauer raised the problem of the Gelfand theory and density of Toeplitz operators in the algebra $\toepalgrad$ over the Bergman space $\Berg$ over the unit ball $\B^n$ in $\C^n$. This question has been answered only for the Bergman space over the unit disc in $\C$ \cite{GMV13,S08}. In the hope of shedding some light on this problem, we prove the density of Toeplitz operators in the algebra $\toepalgrad$ over the Bergman space $\Berg$ using a similar approach.\\

Finally, we show that our method can also be used to prove that Toeplitz operators are SOT dense in bounded operators, giving a new, more constructive, proof of this well-known result of Engli\u{s} \cite{E91}.

\section{Preliminaries}

\subsection{The Fock space} The Gaussian measure $d\lambda=d\lambda_{n}$ is given by  $d\lambda(z)=e^{-\pi|z|^2}dz$, where $dz$ is the Lebesgue measure on $\C^n\simeq \R^{2n}$. The {\it Fock space} $\F$ is a closed subspace of $L^2(\C^n,d\lambda)$, consisting of all holomorphic functions on $\C^n$ that are square integrable with respect to $d\lambda$. It is a reproducing kernel Hilbert space with kernel $K_z$ given by
$$K_z(w)=e^{\pi\Bar{z}w}.$$
The normalized reproducing kernel $k_z$ is given by 
$$k_z(w)=\frac{K_z(w)}{\|K_z\|}=e^{\pi\Bar{z}w-\frac{\pi}{2}|z|^2}.$$
The Bergman projection $P:L^2(\C,d\lambda)\to \F$ is given by
$$Pf(z)=\ip{f}{K_z}=\int f(w)e^{\pi z\Bar{w}}\ d\lambda (w).$$ 
\subsection{Toeplitz operators} For $a\in\bddf$, the Toeplitz operator with symbol $a$, denote $T_a$, is given by $$T_a f(z)=P(af)(z)=\langle af,K_z\rangle; \ \ \forall \ z\in \C^n$$
where $P:\Ltwof \to \F$ is the Bergman projection.  The \textit{Toeplitz algebra}, denoted by $\toepalg$, is the closed sub-algebra of all bounded operators $\bdd$ generated by the set of all Toeplitz operators with bounded symbols. Since adjoints of Toeplitz operators are also Toeplitz operators, $\toepalg$ is a noncommutative $C^*$-algebra.\\

\subsection{Translations and convolutions}\label{sec:convolutions}
In quantum harmonic analysis translations and convolutions are induced by the usual action of $\C^n$ on itself by translations. We denote by $\tau_z$ the usual translations acting on functions and given by
$$\actf{z}a(w)=a(w-z), \ \forall w\in \C^n.$$
For $a\in \bddf$ and $\psi\in L^1(\C^n)$, the usual convolution $\psi\ast a$ is given by
$$\psi\ast a(z):=\int a(z-w)\ \psi(w)dw,\ \ \ z\in \C^n.$$
Then $\psi\ast a$ is a bounded uniformly continuous function on $\C^n$ and $\|\psi\ast a\|_\infty\leq \|\psi\|_1\|a\|_\infty$.

The action of $\C^n$ on the space of bounded operators $\bdd$, denoted $\actop{z}$, is given by the following \textit{operator translations}
$$\actop{z}S=W_zS W_z^*, \ \ S\in \bdd.$$ 
Here $W_z, z\in \C^n$ are the usual Weyl unitary operators on $\F$, given by 
$$W_zf(w)=k_z(w)f(w-z);\ \ w\in \C^n, \ f\in \F.$$
With operator translations being defined, for $\psi\in L^1(\C^n)$ and $S\in \bdd$, the \textit{operator convolution} $\psi\ast S$ is defined by
$$S\ast \psi :=\psi\ast S:=\int   \actop{z} S \ \psi(z)dz.$$
Similarly as for functions, $\|\psi\ast S\|\leq \|S\|\|\psi\|_1$, and each $\psi\ast S$ is uniformly continuous, i.e., the map $z\to \actop{z}(\psi\ast S)$ is continuous in operator norm. We denote by $\cbuop$ the algebra of bounded, uniformly continuous operators. 

We have the following basic properties of operator translations and convolutions. We point to \cite{W84, LS18, F19} for details. For $\psi,\psi_1,\psi_2\in L^1(\C^n)$ and $S\in\bdd$,
    \begin{enumerate}[label=(\roman*)]
        \item The map $z\to \actop{z}S$ is continuous in SOT.
        \item $\actop{z}(\psi\ast S)=(\actf{z} \psi)\ast a =\psi\ast (\actop{z} S) $.
        \item $\psi_1\ast (\psi_2\ast S)=(\psi_1\ast \psi_2)\ast S=(\psi_2\ast \psi_1)\ast S$.
    \end{enumerate}  
    
Additionally, one also defines an operator convolution between two operators in the following way. Let $\tclass$ denote the set of all trace class operators on $\F$.  Let $T$ and $S$ are two bounded operators, at least one of which is in trace class $\tclass$. One defines $T\ast S$ as a function $\C^n\to \C$ given by \[T\ast S(z):=\text{Tr}(TL_z(USU))\]
where $U$ is the self-adjoint operator on $\F$ given by
$$(Uf)(z)=f(-z),\ \ \ z\in \C^n,\ f\in \F.$$
Then $T\ast S\in \cbuC$ and $\|T\ast S\|_\infty\leq \|T\|\|S\|_1$ if e. g. $S\in \tclass$. Furthermore, if both $T, S\in  \tclass$, then $T\ast S\in L^1(\C^n)$.
For $A,C\in\bdd$, $B\in \tclass$, the following basic identities hold
    \begin{enumerate}[label=(\roman*)]
        \item $A\ast B= B\ast A$.
        \item $(A\ast B)\ast C=A\ast (B\ast C)$. 
    \end{enumerate}  

Finally, one defines the convolution $a\ast S$ for $a\in\bddf$ and $S\in \tclass$ by duality:
$$\ip{a\ast S}{A}_{tr}=\ip{a}{A\ast(USU)}_{tr}\ \ \ \forall A\in \tclass$$
where $\ip{A}{B}_{tr}=Tr(AB)$ and $\ip{f}{g}_{tr}=\int_{\C^n}f(z)g(z) dz$. Note that $a\ast S$ can be defined weakly as well, and the two definitions coincide. We have $a\ast S\in \cbuop$ and $\|a\ast S\|\leq \|a\|_\infty \|S\|_1$. Also if $A,B\in \tclass$ and $a\in \bddf$, we have
$$a\ast (A\ast B)=(a\ast A)\ast B.$$

\subsection{Toeplitz operators as convolutions} We denote by $\Phi\in \bdd$ the rank operator $\Phi:=k_0\otimes k_0$. For $a\in \bddf$, due to $L_z\Phi=k_z\otimes k_z$, it is easy to see that \[T_a=a\ast \Phi.\] This simple observation establishes the main connection between Toeplitz operators and quantum harmonic analysis. The following relations are then immediate consequences of the basic properties of convolutions: for $a\in L^\infty(G)$ and $\psi\in L^1(G)$, $\actop{z}(T_a)=T_{\actf{z} a}$ and  $\psi\ast T_a= T_{\psi \ast a}$

Since $a\ast \Psi\in \cbuop$ we immediately see that the set of all Toeplitz operators with bounded symbols, and consequently the whole Toeplitz algebra $\toepalg$ is contained in $\cbuop$. It was proved by Fulsche ~\cite{F19} that these two $C^*$ algebras coincide. We provide another argument for this fact in Section \ref{sec:cbu}.

\subsection{Berezin transform as a convolution} Recall that for $a\in\bddf$ the Berezin transform of $a$, denoted $B(a)$, is a function on $\C^n$ defined by
$$B(a)(z)=\ip{ak_z}{k_z}, \ z\in\C^n,$$
where $k_z$'s are the normalized reproducing kernels in $\F$.
Similarly, for $S\in\bdd$, the Berezin transform of the operator $S$, denoted $B(S)$, is a function on $\C^n$ defined by
$$B(S)(z)=\ip{Sk_z}{k_z}, \ z\in\C^n.$$
Clearly, $B(T_a)=B(a).$

We will denote by $\varphi$ the standard Gaussian function $$\varphi(z)=e^{-\pi|z|^2}=|\ip{k_0}{k_z}|^2,\ \ z\in \C^n.$$ Clearly, $\varphi=\Phi\ast \Phi$. It is not hard to check that 
\[B(a)=\varphi\ast a, \text{ and } B(S)=\Phi\ast S.\] Therefore, both $B(a)$ and $B(S)$ are uniformly continuous functions. Furthermore, since $\varphi=\Phi\ast \Phi$ we also have that 
\[\varphi\ast S=(\Phi\ast \Phi)\ast S=(\Phi\ast S)\ast \Phi= T_{B(S)},\] for every bounded operator $S$.

 \subsection{Action of the affine unitary group on the Fock space.}
Here we introduce the $G$-convolutions for subgroups of the affine unitary group. The group of all $n\times n$ unitary matrices $U_n$ act on functions on $\C^n$ via 
$$\actf{A}f(z)=f(A^{-1}z),\ \ \ z\in \C^n, A\in U_n.$$
and $A\mapsto \actf{A}$ defines a unitary representation of $U_n$ on $\F$.

 The semi-direct product of $U_n$ and $\C^n$, denoted $U_n\ltimes \C^n$ acts on $\C^n$ by rotations followed by translations:
 $$(A,z)\cdot w=Aw+z;\ \ \ z,w\in \C^n, A\in U_n.$$
 The unimodular group $U_n\ltimes \C^n$ has Haar measure given by
 $$d\mu_{U_n\ltimes \C^n}((A,z))=d\mu_{U_n}(A) dz.$$
  Since $\C^n$ is abelian, the Weyl representation of $\C^n$ and the representation of $U_n$ together induces a projective representation $\pi$ of $U_n\ltimes \C^n$ on $\F$, given for $g=(A_g,z_g)\in U_n\ltimes \C^n$ by
  $$\pi(g)f=j(g^{-1},\cdot)\actf{g}f, \ \ f\in \F$$
  where $j(g,\cdot)=k_{A_g^{-1}z_g}$ and $\actf{g}f(w)=f(g^{-1}w)=f(A_g^{-1}(w-z_g))$ for $w\in \C^n$.\\

Let $G$ be a subgroup of $U_n\ltimes \C^n$ and let $\piG$ be the restriction of the projective representation $\pi$ of $U_n\ltimes \C^n$ to $G$. Then $\piG$ also gives rise to a representation $\TpiG$ of $G$ acting on the Banach space $\bdd$, given by
$$\TpiG(g)S=\piG(g)S\piG(g)^*;\ \ \ S\in\bdd, \ g\in G.$$
We call $\TpiG(g)S$ the translation of $S$ by $g$ and is often denoted by $\actop{g}S$.
An operator $S\in\bdd$ is said to be $G$-uniformly continuous if $g\mapsto \TpiG(g) S$ is continuous with respect to the norm topology. The set of all $G$-uniformly continuous bounded operators, denoted $\cbuopG$ is a $C^*$-subalgebra of $\bdd$ containing the set of all compact operators.

 Let $L^1(G)$ denote the set of all functions on $G$ that are integrable w.r.t. the Haar measure $\mu:=\mu_G$. The $G$-convolution of $\psi\in L^1(G)$ and $S\in \bdd$, is given by
$$\Avg S=\int_G  \actop{g}S\ \psi(g) \haar{g}$$
where the integral is understood in the weak sense. Then 
$$\|\psi \ast_G S\|\leq \|S\| \|\psi\|_1$$
and in addition, $\Avg S$ is in $\cbuopG$.
By the above inequality, convolutions $S\mapsto \Avg S$ are continuous in norm topology. 

Also, we have the $G$-convolution of $\psi\in L^1(G)$ and $a\in \bddf$ given by
$$\Avg a (z)=\int_G a(g^{-1}\cdot z)\ \psi(g) \haar{g};\ \ z\in \C^n.$$
In addition, we get that,
$$\Avg T_a=T_{\Avg a}$$
where $T_a$ is the Toeplitz operator on $\F$ with symbol $a$. We explore this in detail in later sections. When $G=\C^n$ we get the usual convolutions.

\begin{remark}
    The symplectic group $Sp_n$ also acts on the Fock space via the double-valued metaplectic representation. Hence one could instead consider the affine symplectic group $Sp_n\ltimes \C^n$ mentioned in \cite{W84}. However, the metaplectic representation is not necessarily given by a geometric action on $\C^n$, unlike with $U_n$, which will be required later on. Hence we restrict our attention to $U_n$ which is isomorphic to $Sp_n\cap O_{2n}$. Nevertheless, it is an interesting question whether subgroups of $Sp_n$ exist, other than subgroups of $U_n\times\C^n$, where the representation is given in terms of the group’s action on $\C^n$.
\end{remark}

\subsection{\texorpdfstring{$G$}{G}-invariant functions and operators} Let $G<U_n\ltimes\C^n$. We say a function $a$ on $\C^n$ is $G$-invariant if $a(g^{-1}z)=z$ for all $z\in \C^n$ and $g\in G$. We denote the set of all $G$-invariant bounded functions by $\homf$.

We say $S\in \bdd$ is $G$-invariant (or an intertwining operator) if $\TpiG(g)S=S$ for all $g\in G$. Let $\Hom$ denote the set of all $G$-invariant (intertwining) operators. Then $\Hom$ is a closed subalgebra of $\bdd$.



\section{The Algebra of Bounded Uniformly Continuous Operators}\label{sec:cbu}

In this section, we go over Fulsche's correspondence theory and a quantum version Wiener's Tauberian theorem.

\subsection{Fulsche's correspondence theory}
The correspondence theory for $p$-Fock spaces was discussed in \cite{F19}. Here we briefly discuss correspondence theory focusing on $p=2$.
Let $\cbuC$ denote the $C^*$-subalgebra of $\bddf$ given by the set of all bounded uniformly continuous functions on $\C^n$. Recall that the set of all \textit{uniformly continuous operators} $\cbuop$ is given by
$$\cbuop= \{S\in\bdd \mid z\to \actop{z}S \text{ is continuous in } \|\cdot\| \}$$

Then $\cbuop$ forms a $C^*$-subalgebra of $\bdd$ and Fulsche gave the following characterizations of $\cbuop$ in \cite{F19}:
 \begin{align}\label{eq:Fulsche correspondence}
      \begin{split}
          \cbuop&= \overline{\text{span}\{\psi\ast S \mid \psi\in L^1(\C^n), S\in \bdd   \}}\\
      &= \{S\in \bdd \mid \psi_t \ast S \to S \text{ in norm} \}\\
      &=\overline{\{T_a \mid a\in \cbuC  \}}\\
      &=\overline{\bddf \ast \Phi}
      \end{split} 
  \end{align} 
  where  $\{\psi_t\}$ is an approximate identity for $L^1(\C^n)$. 
  
   Since the set of all Toeplitz operators with bounded symbols is contained in $\cbuop$, as an application of correspondence theory, the Toeplitz algebra $\toepalg$ coincides with $\cbuop$:
\begin{align}\label{eq:correspondence_Toeplitz}
      \cbuop=\toepalg.
  \end{align}

\subsection{Wiener's Tauberian theorem}
There are several equivalent ways to define regularity for $L^1$ functions, the most classical definition being that the Fourier transform of $\psi$ vanishes nowhere. The classical Wiener's Tauberian theorem says that a function $\psi\in L^1(\C^n)$ is regular if and only if the translates of $\psi$ are dense in $L^1(\C^n)$, which is also equivalent to having
$$L^1(\C^n)=\overline{\psi\ast L^1( \C^n) }.$$
In the papers \cite{W84, LS18, FG23}
the notion of regularity was extended to operators.
An operator $\Psi\in \tclass$ is said to be {\it regular} if translates of $\actop{z}\Psi$ are dense in $\tclass$. The following are two equivalent forms of regularity that we will use in our proofs below (see e.g.~\cite{LS18, FG23} for a proof). 
\begin{enumerate}[label=(\roman*)]
    \item $\Psi$ is a regular operator.
    \item $\Psi \ast \Psi$ is a regular function.
    \item $\tclass=\overline{\Psi \ast \tclass}$.
\end{enumerate}
Note that the Gaussian function $\phi$ is regular, as is the operator $\Phi=k_0\otimes k_0$.
It is not hard to see that a convolution of two regular functions or two regular operators yields a regular function. Similarly, convolving a regular operator with a regular function gives a regular operator.

For any regular function $\psi\in L^1(\C^n)$ we have
\begin{align}\label{eq:Wiener's theorem cbu}
    \cbuC =\overline{\psi\ast \bddf}=\overline{\psi\ast \cbuC}.
\end{align}
as a consequence of Wiener's theorem. The following quantum version of this fact is closely related to quantum forms of Wiener's theorem discussed in \cite{W84, LS18, FG23}. However, since we weren't able to find it explicitly stated in the literature, for completeness, we include its simple proof. 

\begin{proposition}\label{prop:Wiener's theorem}
Let $\psi\in L^1(\C^n)$ be a regular function and let $\Psi\in \tclass$ be a regualar operator. Then
\begin{enumerate}
    \item $\cbuC =\overline{\psi\ast \bddf}=\overline{\psi\ast \cbuC}$
    \item $\cbuop=\overline{\psi\ast \bdd}=\overline{\psi\ast \cbuop}$
    \item $\cbuC =\overline{\Psi\ast \bdd}=\overline{\Psi\ast \cbuop}$
    \item  $\cbuop=\overline{\Psi\ast \bddf}=\overline{\Psi\ast \cbuC}$.
\end{enumerate}
In particular,
 $$\cbuop=\overline{\psi \ast \toep}=\overline{\{ T_{\psi \ast a}\mid a\in \bddf\}}$$
 where $\toep$ denotes the set of all Toeplitz operators with bounded symbols:
 $$\toep=\{T_a\mid a\in \bddf\}.$$
\end{proposition}

\begin{proof}
    Proofs of (1) and (2) are similar. We prove (2).
    Since $\psi\ast S$ is a uniformly continuous operator for any $S\in \bdd$, we have $$ \psi\ast \cbuop \subset \psi\ast \bdd\subset \cbuop.$$
    To prove $\cbuop\subset \overline{\psi\ast \cbuop} $, let $S\in \cbuop$ and let $\epsilon>0$. Let $\{f_t\}$ be an approximate identity for $L^1(\C^n)$. Then $f_t\ast S\to S$ and hence there exists $t_0$ s.t. 
    $$\|S-\psi_{t_0}\ast S\|<\frac{\epsilon}{2}.$$
    By Wiener's Tauberian theorem, there exists $b\in L^1(G)$, s.t. 
    $$\|\psi_{t_0}-\psi\ast b\|_1<\frac{\epsilon}{2(\|S\|+1)}.$$
    Then $b\ast S\in \cbuop$ and
    \begin{align*}
        \|S-\psi\ast (b\ast S)\| &\leq \|S-\psi_{t_0} \ast  S\|+\|\psi_{t_0} \ast  S-(\psi\ast b)\ast S\|\\
        &\leq \frac{\epsilon}{2}+\|\psi_{t_0}-\psi\ast b\|_1\|S\|\\
        &<\epsilon
    \end{align*}
    as required.\\
    Proof of (3): Note that $\Psi\ast \Psi$ is a regular function. Then by properties of convolutions and statement (2), we have
    \begin{align*}
        \overline{\Psi\ast \cbuop} &\subset \overline{\Psi\ast \bdd}\\
        &\subset \cbuop\\
        &=\overline{(\Psi\ast \Psi)\ast \cbuop}\hspace{1cm} \text{ (by (2))}\\
        &\subset \overline{\Psi\ast (\Psi\ast \cbuop)}\\
        &\subset \overline{\Psi\ast \cbuop}
    \end{align*}
    proving statement (3). The proof of (4) is similar.\\
    To prove that last statement, recall that for any $a\in \bddf$ we have $T_a=\Phi\ast a$, where $\Phi=k_0\otimes k_0$. Since $\psi$ is a regular function and $\Phi$ is a regular operator we have that $\psi\ast \Phi$ is also a regular operator. Therefore, using (2), we obtain
    \[\cbuop=\overline{(\psi\ast\Phi)\ast \bddf}=\overline{\psi\ast\toep}.\]
\end{proof}

Since the Gaussian $\varphi$ is regular, we have
\begin{align}
    \begin{split}
        \cbuop&= \overline{\varphi\ast \bdd}\\
    &=\overline{\{T_{B(S)} \mid S\in\bdd\}}.
    \end{split}
\end{align}
 Also, we get the following corollary which was already proved in \cite{F19}.

\begin{corollary}  
        The following equality holds:
        $$\cbuop=\overline{\{T_a \mid a\in \cbuC  \}}.$$   
\end{corollary}


\section{\texorpdfstring{$G$}{G}-invariant Toeplitz algebras}\label{sec:G-invariant}

In this section we characterize certain $G$-invariant Toeplitz algebras. Recall that $U_n\ltimes \C^n$ acts on the Fock space via the and the projective representation $\pi$ induced by the representation of $U_n$ and the Weyl representation of $\C^n$. Let $G=\Grota\ltimes \Gtrans$ ($\Grota<U_n$ and $\Gtrans<\C^n$) be a subgroup of $U_n\ltimes \C^n$ equipped with Haar measure $d\mu:=d\mu_G$. Let $\piG$ be the restriction of $\pi$ to $G$. The set of all $\piG$-intertwining operators, denoted $\Hom$ is a von Neuman subalgebra of $\bdd$. Let $\toepalgG$ denote the norm closed subalgebra of $\bdd$ generated by the set of all $G$-invariant Toeplitz operators. 

The density of Toeplitz operators with symbols in $\cbuC$ in $\toepalgG$ for the case $G<\C^n$ was discussed in \cite{F19}. His approach relied on the invariance of $\toepalgG$ under the translations $L_z$, which is not satisfied by radial Toeplitz algebras. We discuss the same for the more general case $G<U_n\ltimes \C^n$ which includes radial and quasi-radial Toeplitz algebras. Note that in the case $G<U_n$ averaging over the representation can be used to prove density theorems (see \cite{DOQ15}) for $G$-invariant Toeplitz operators. This is no longer possible for groups $G$ for which $L^1(G)$ doesn't contain constant functions. We take a unified approach to prove that $G$-invariant Toeplitz algebras are linearly generated by Toeplitz operators. And more importantly, we prove the following about the $G$-invariant Toeplitz algebra $\toepalgG$ in Theorem \ref{theo:density_subToeplitzalgebras}:
$$\toepalgG=\overline{\{T_{B(S)} \mid S\in \Hom\}}.$$

 Recall that the $G$-convolution of $\psi\in L^1(G)$ and $a\in \bddf$, is given by
$$(\Avg a)(z)=\int_G  a(g^{-1}\cdot z)\ \psi(g)\haar{g};\ \ \ z\in\C^n$$
and the $G$-convolution of $\psi\in L^1(G)$ and $S\in \bdd$, is given by
$$\Avg S=\int_G  \actop{g}S \ \psi(g)\haar{g}.$$
Then the map $S\mapsto \psi\ \ast_G S$ is continuous and $\|\psi \ast_G S\|\leq \|\psi\|_1 \|S\| .$
However, we also have the continuity of convolutions in the strong operator topology of $\bdd$ and this plays a part in the proofs that follow.

\begin{proposition}\label{prop:SOTconvergence}
    Let $S_k,S\in\bdd$ s.t. $S_k\to S$ in strong operator topology and let $\psi\in L^1(G)$. Then  $$\Avg S_k\to \Avg  S$$ in strong operator topology.
\end{proposition}

\begin{proof}
Since $S_k\to S$ in SOT, there exists $C>0$ s.t. $\|S_k\|\leq C$ by uniform boundedness principle.
    Note that for $f\in \F$,
    \begin{align*}
        \|(\psi &\ast S_k)f-(\psi \ast S)f\|^2= \ip{(\psi \ast S_k)f-(\psi \ast S)f}{(\psi \ast S_k)f-(\psi \ast S)f}\\
        &=\|(\psi \ast S_k)f\|^2-\ip{(\psi \ast S_k)f}{(\psi \ast S)f}-\ip{(\psi \ast S)f}{(\psi \ast S_k)f}+\|(\psi \ast S)f\|^2\\
    \end{align*}
    Also, for $f\in\F$,
    \begin{align*}
        \|(\psi \ast S_k)f\|^2&=\ip{(\psi \ast S_k)f}{(\psi \ast S_k)f}\\
        &=\int \ip{L_zS_kf}{(\psi \ast S_k)f} \psi (z) d\lambda(z)\\
        &=\int \overline{\ip{(\psi \ast S_k)f}{L_zS_kf}} \psi (z) d\lambda(z)\\
        &= \int \int \ip{L_zS_kf}{L_wS_kf} \overline{\psi (w)} \psi (z) d\lambda (w) d\lambda(z).\\
    \end{align*}
Then since we have 
\begin{align*}
    |\ip{L_zS_kf}{L_wS_kf}&-\ip{L_zSf}{L_wSf}|\\
    &\leq |\ip{L_z(S_k-S)f}{L_wS_kf}|+|\ip{L_zSf}{L_w(S_k-S)f}|\\
    &\leq \|L_z(S_k-S)f\|\|L_wS_kf\|+\|L_zSf\|\|L_w(S_k-S)f\|\\
    &\leq \|(S_k-S)\pi(z)^*f\|C\|f\|+ \|S\|\|f\|\|(S_k-S)\pi(w)^*f\|,
\end{align*}
it follows that
$$\lim_{k} \ \ip{L_zS_kf}{L_wS_kf} = \ip{L_zSf}{L_wSf}$$
by the strong convergence of $S_k\to S$.
And also,
$$|\ip{L_zS_kf}{L_wS_kf}|\leq C^2\|f\|^2.$$
Therefor by Lebesgue dominated convergence theorem $$\lim_{k} \|\psi\ast S_kf\|=\|\psi\ast Sf\|.$$

By a similar application of the dominated convergence theorem and the fact that $S_k\to S$ in weak operator topology, we get that $$\lim_k \ip{(g\ast S_k)f}{(g\ast S)f}=\ip{(\psi\ast S)f}{(\psi\ast S)f}=\|(\psi\ast S)f\|^2.$$

Hence $$\lim_k \|(\psi\ast S_k)f-(\psi\ast S)f\|=0.$$   
\end{proof}

Our main theorem is the following $G$-invariant version of the Wiener's Tauberian theorem:

 \begin{theorem}\label{theo:density_subToeplitzalgebras}
     Let $G$ be a subgroup of $\U_n\ltimes \C^n$. Denote the set of all $G$-invariant bounded uniformly continuous functions by $\cbuCGinvar$ and let $$\cbuopGinvar=\cbuop\cap \Hom.$$
     Suppose $\psi\in L^1(\C^n)$ is a radial regular function and let $\Psi$ be a regular radial operator. 
     Then 
     \begin{enumerate}
         \item $\cbuCGinvar=\overline{\psi\ast \cbuCGinvar}$
         \item $\cbuopGinvar=\overline{\psi\ast \cbuopGinvar}$
         \item $\cbuCGinvar=\overline{\Psi\ast \cbuopGinvar}$
         \item  $\cbuopGinvar=\overline{\Psi\ast \cbuCGinvar}$.
    \end{enumerate}
    In particular, we have,
    $$\toepalgG=\cbuopGinvar=\overline{\{T_a\mid a\in \cbuCGinvar\}}.$$
 \end{theorem}

 The proof of Theorem \ref{theo:density_subToeplitzalgebras} requires several lemmas.

 \subsection{\texorpdfstring{$G$}{G}-convolutions of Toeplitz operators}
 First, we show that convolutions preserve Toeplitz operators. 
The function $a\in \bddf$ is said to be $G$-invariant if $\avg a=a$.

 \begin{lemma}
     Let $a\in\bddf$. Then
     $$\actop{g}T_a=T_{\actf{g} a}.$$
 \end{lemma}

 \begin{proof}
     Recall that for $g=(A,z)$, $j(g,w)=k_{A^{-1}z}(w)=e^{-\pi\overline{A^{-1}z}w-\frac{\pi|z|^2}{2}}$. For $f_1,f_2\in \F$, we have
     \begin{align*}
         \ip{\actop{g}T_af_1}{f_2}&=\ip{M_a\piG(g^{-1})f_1}{\piG(g^{-1})f_2}\\
         &=\int a(w)|j(g,w)|^2f_1(g\cdot w)\overline{f_2(g\cdot w)} d\lambda(w) \\
         &=\int a(w)f_1(g\cdot w)\overline{f_2(g\cdot w)} e^{-\pi|A^{-1}z+w|^2} dw \\
         &=\int a(w)f_1(Aw+z)\overline{f_2(Aw+z)} e^{-\pi|z+Aw|^2} dw\\
         &= \int a(A^{-1}(w-z))f_1(w)\overline{f_2(w)}\ d\lambda(w)\\
         &= \ip{T_{\actf{g} a}f_1}{f_2} 
     \end{align*}
     as required.
 \end{proof}

 \begin{lemma}\label{lem:Toeplitz_invariance}
     Let $a\in\bddf$. 
     $$(\Avg T_a)=T_{\avg a}.$$
 \end{lemma}

 \begin{proof}
     Notice that for $f_1,f_2\in \F$,
     \begin{align*}
         \ip{(\Avg T_a)f_1}{f_2}&=\int_G \ip{\piG(g)T_a\piG(g)^*f_1}{f_2} \psi(g)\haar{g}\\
         &=\int_G \ip{T_{g\cdot a}f_1}{f_2} \psi(g)\haar{g}\\
         &=\int_G \int_{\C^n} a(g^{-1}\cdot w)f_1(w)\overline{f_2(w)} d\lambda(w) \ \psi(g)\haar{g}\\
         &=\int_{\C^n} \int_G  a(g^{-1}\cdot w) \psi(g)\haar{g}\ f_1(w)\overline{f_2(w)}d\lambda(w)\\
         &=\ip{T_{\avg a}f_1}{f_2}.
     \end{align*}
     as required.
 \end{proof}

 \subsection{\texorpdfstring{$G$}{G}-convolutions vs usual convolutions}
 The proof of Theorem \ref{theo:density_subToeplitzalgebras} has elements similar to classical radialization. However, due to the possible non-compactness of $G$, averaging over the representation cannot be applied. A key ingredient of our proof is
the commutativity of $G$-convolutions with usual convolutions of radial $L^1$ functions (or trace class operators). 

 \begin{lemma}\label{lem:convolution_and_gconvolution}
     Let $a\in \bddf$, $S\in \bdd$, $\psi\in L^1(G)$ and let $\Psi\in \tclass$. Also let $h\in L^1(\C^n)$ be a radial function and $H\in \tclass$ be a radial operator. Then    
     \begin{enumerate}
         \item $\avg (h\ast a)=h\ast(\avg a)$
         \item $\Avg (h \ast S)=h\ast (\Avg S)$
         \item $\avg (H\ast a)=H\ast(\avg a)$
         \item $\Avg (H \ast S)=H\ast (\Avg S)$.
     \end{enumerate}
 \end{lemma}

 \begin{proof}
     Notice  that for $g=(A_g,z_g)$
     \begin{align*}
         \actf{g}(h\ast a) (z)
         &= (a\ast h) (A_g^{-1}(z-z_g))  \\
         &= \int_{\C^n} a(w) h(A_g^{-1}(z-z_g)-w)\ dw\ \\
         &= \int_{\C^n} a(w) h(A_g(A_g^{-1}(z-z_g)-w))\ dw  \\
         &\hspace{5cm} (\text{as $h$ is radial})\\
         &=\int_{\C^n} a(w) h(z-z_g-A_gw)\ dw\\
         &=\int_{\C^n} a(A_g^{-1}(w-z_g)) h(z-w)\ dw\\
         &=\int_{\C^n} a(g^{-1}\cdot w) h(z-w)\ dw
     \end{align*}
     by several changes of variables. Hence by Fubini's theorem and commutativity of usual convolutions,
     \begin{align*}
         (\avg (h\ast a))(z)&=\int_G \actf{g}(h\ast a) (z) \psi (g) \haar{g}\\
         &= \int_G \int_{\C^n} a(g^{-1}\cdot w) h(z-w)\ dw\ \psi (g) \haar{g}\\
         &= \int_{\C^n} (\avg a)(w) h(z-w)\ dw\\
         &= ((\avg a) \ast h(z)\\
         &= (h \ast (\avg a))(z).
     \end{align*}
     
     To prove the statement (2) and (3), notice that for a Toeplitz operator $T_a$, we have that
     \begin{align*}
        \Avg (h\ast T_a)=\Avg T_{h\ast a}=T_{\avg(h\ast a)}=T_{h\ast(\avg a)}=h\ast (\Avg T_{a}).
    \end{align*}
    and 
    \begin{align*}
        \Avg (T_h\ast a)=\Avg T_{h\ast a}=T_{\avg(h\ast a)}=T_{h\ast(\avg a)}=T_h\ast (\Avg a).
    \end{align*}
    by what was proved above and by Lemma \ref{lem:Toeplitz_invariance}. Hence the result holds for $S\in \bdd$ by the density of Toeplitz operators in $\bdd$ and the continuity of convolutions in the strong operator topology (Lemma \ref{prop:SOTconvergence}).
    For the proof of (4), we use (3). Note that by commutativity of usual convolutions and by applying (3) to the radial operator $\Phi$, we have 
    \begin{align*}
        (\psi \Gast (H\ast T_a))\ast \Phi &= \psi \Gast ((H\ast T_a) \ast \Phi)\\
        &= \psi \Gast (H \ast (a\ast \varphi))\\
        &= H \ast (\psi\Gast (a\ast \varphi))\hspace{1cm} \text{(by (3))}\\
        &= H \ast (\psi\Gast (T_a \ast \Phi))\\
        &= H \ast \Phi \ast (\psi \Gast T_a) \hspace{1cm} \text{(by (3))}\\
        &= (H \ast  \Phi \ast (\psi \Gast T_a)) \ast \Phi.
    \end{align*}
    Therefore  $\Avg (H \ast S)=H\ast (\Avg S)$ by the regularity of $\Phi$.
 \end{proof}


\subsection{\texorpdfstring{$G$}{G}-invariance via convolutions}

 \begin{lemma}\label{lem:convolution_invariance}
     Let $a\in \bddf$ and $S\in\bdd$. Then 
     \begin{enumerate}
         \item $a\in\homf$ iff $\Avg a=(\int_G \psi \ d\mu)a$ for all $\psi\in L^1(G)$.
         \item $S\in\Hom$ iff $\Avg S=(\int_G \psi \ d\mu)S$ for all $\psi\in L^1(G)$.       
     \end{enumerate}
 \end{lemma}

 \begin{proof}
     Proofs of (1) and (2) are similar. We prove only (2). Assume $S\in \Hom$. Then for $f_1,f_2\in\F$
     \begin{align*}
         \ip{(\Avg S)f_1}{f_2}&=\int_G \ip{\actop{g}Sf_1}{f_2} \psi(g)\haar{g}\\
         &= \int_G \ip{Sf_1}{f_2} \psi(g)\haar{g}\\
         &= \ip{Sf_1}{f_2} \int_G \psi(g)\haar{g}
     \end{align*}
     and hence $\Avg S=(\int_G \psi \ d\mu)S$.
     For the other implication assume $\Avg S=(\int_G \psi \ d\mu)S$ for all $\psi\in L^1(G)$. Let $f_1,f_2\in \F$. Then for all $\psi\in L^1(G)$
     \begin{align*}
         \int_G \ip{(\actop{g}S-S)f_1}{f_2} \psi(g)\haar{g}
         &=\ip{(\Avg S)f_1}{f_2}-\int_G \psi(g)\haar{g}\ip{Sf_1}{f_2}\\
         &=0.
     \end{align*}
     Then since $g\to\ip{(\actop{g}S-S)f_1}{f_2}$ is a continuous function in $\bddfG$, we have $\ip{(\actop{g}S-S)f_1}{f_2}=0$ for all $g\in G$. Therefore $\actop{g} S=S$ for all $g\in G$ and $S\in \Hom$.
 \end{proof}

 \begin{lemma}\label{lem:hom_radial_convol}
     Let $S\in \Hom$ and $a\in \bddf$.  $h\in L^1(\C^n)$ be a radial function and let $H\in \tclass$ be a radial operator. Then 
     \begin{enumerate}
         \item $h\ast S\in \Hom$
         \item $H \ast a \in \Hom$
         \item $H\ast S\in \homf$.
     \end{enumerate}
 \end{lemma}

 \begin{proof}
      Note that for all $\psi\in L^1(G)$,
      \begin{align*}
          \Avg (h\ast S)&= h\ast (\Avg S)\ \ \ \  \text{(by Lemma \ref{lem:convolution_and_gconvolution})}\\
          &= h\ast \Big(\Big(\int \psi \ d\mu \Big) S\Big) \ \ \ \ \text{(by Lemma \ref{lem:convolution_invariance})}\\
          &= \Big(\int \psi \ d\mu\Big)(h\ast S)
      \end{align*}
      proving (1).
      Similarly, for all $\psi\in L^1(G)$,
      \begin{align*}
          \Avg (H\ast a)= H\ast (\Avg a) = \Big(\int \psi \ d\mu\Big)(H\ast a)
      \end{align*}
      and 
      \begin{align*}
          \psi \Gast (H\ast S)= H\ast (\Avg S)= \Big(\int \psi \ d\mu\Big)(H\ast S)
      \end{align*}
      by lemmas \ref{lem:convolution_and_gconvolution}, \ref{lem:convolution_invariance} and by properties of convolutions discussed in Subection \ref{sec:convolutions}.
 \end{proof}

 \begin{lemma}\label{lem:hom_phi_convol}
     Let $S\in \bdd$. Then  $S\in \Hom$ iff $\conv{S}\in \Hom$.
 \end{lemma}

 \begin{proof}
     Due to Lemma \ref{lem:hom_radial_convol}, only need to show the converse. Assume $\conv{S}\in \Hom$. Then for all $\psi\in L^1(G)$,
     $$\varphi \ast (\Avg S)=\Avg (\varphi \ast S)=\Big(\int \psi \ d\mu\Big)(\varphi \ast S)= \varphi \ast \Big(\Big(\int \psi \ d\mu\Big) S\Big).$$
     Since $S\mapsto \conv{S}$ is injective,
     $$\Avg S= \Big(\int \psi \ d\mu\Big)S \ \ \forall \psi\in L^1(G)$$
     and hence $S\in \Hom$ by Lemma \ref{lem:convolution_invariance}.
 \end{proof}

 \subsection{The proof of Theorem \ref{theo:density_subToeplitzalgebras}}
 Now we present the proof of Theorem \ref{theo:density_subToeplitzalgebras}.

The following lemma plays an important role in the proof and can be found in \cite{RS00} (Lemma 1.4.2).
\begin{lemma}[Wiener's division lemma]\label{lem:division}
    Let $\psi,f\in L^1(\R^n)$ s.t. $\psi$ is regular and $\hat{f}$ has compact support. Then there exists a function $h\in L^1(\R^n)$ s.t. $f=\psi \ast h$.
\end{lemma}

\begin{proof}
    
    Proofs of (1) and (2) are similar. We provide a proof of (2). The inclusion $\overline{\psi \ast \cbuopGinvar} \subset \cbuopGinvar$ is immediate because for $S\in \Hom$, $\psi \ast S\in \Hom$ by Lemma \ref{lem:hom_phi_convol}. 
    
    To prove $\cbuopGinvar\subset \overline{\psi\ast\cbuopGinvar }$, let $S\in \cbuopGinvar$. Also, let $\{f_t\}$ be an approximate identity of radial functions s.t. $\hat{f_t}$ are compactly supported (such $f_t$ can be constructed by dilating the Fourier inverse of the bump function). Then since $S\in \cbuop$, 
    $$\lim_{t\to 0^+} f_t \ast S\to S.$$
    By Wiener's division lemma \ref{lem:division}, there is $h_t\in L^1(\C^n)$ s.t. $f_t=\psi \ast h_t$, i.e. 
    $$\hat{h}_t=\frac{\hat{f_t}}{\hat{\psi}}.$$
    Since $\frac{\hat{f_t}}{\hat{\psi}}$ is a radial function, $h_t$ is also a radial function. Therefore, we have
    $$\lim_{t\to 0^+} \psi\ast (h_t \ast S)\to S.$$
    Note that $h_t \ast S\in \cbuopGinvar$ by Lemma \ref{lem:hom_radial_convol} and hence $S\in  \overline{\psi \ast {\cbuopGinvar}}$. 
    
    Proof of (3): Note that if $S\in \cbuopGinvar$ then $\Psi\ast S\in \cbuCGinvar$  by Lemma \ref{lem:hom_radial_convol}. Also, $\Psi \ast \Psi$ is a regular function. Therefore, 
    \begin{align*}
        \overline{\Psi \ast \cbuopGinvar} &\subset \cbuCGinvar \hspace{2cm} \text{( by Lemma \ref{lem:hom_radial_convol})}\\
        &= \overline{(\Psi \ast \Psi) \ast \cbuCGinvar}  \ \ \ \text{ (by statement (1))}\\
        &= \overline{\Psi \ast (\Psi \ast \cbuCGinvar)}\\
        &\subset \overline{\Psi \ast \cbuopGinvar} \hspace{1.6cm} \text{( by Lemma \ref{lem:hom_radial_convol})}
    \end{align*}
    and hence 
    $$\overline{\Psi \ast \cbuopGinvar}=\cbuCGinvar.$$

    Proof  of (4): Similarly by  Lemma \ref{lem:hom_radial_convol} and statement (2), we have 
    \begin{align*}
        \overline{\Psi \ast \cbuCGinvar} &\subset \cbuopGinvar \hspace{2cm} \text{( by Lemma \ref{lem:hom_radial_convol})}\\
        &= \overline{(\Psi \ast \Psi) \ast \cbuopGinvar}  \ \ \ \text{ (by statement (2))}\\
        &= \overline{\Psi \ast (\Psi \ast \cbuopGinvar)}\\
        &\subset \overline{\Psi \ast \cbuCGinvar} \hspace{1.5cm} \text{( by Lemma \ref{lem:hom_radial_convol})}
    \end{align*}
    proving
    $$\overline{\Psi \ast \cbuCGinvar}=\cbuopGinvar.$$
     
    We obtain the last equality by taking $\Psi=\Phi=k_0\otimes k_0$ in (4). 
\end{proof}

 
 \section{Quasi-radial Toeplitz algebra}\label{sec:example-qrad}
 In this section, we investigate the density of Toeplitz operators in the quasi-radial Toeplitz algebra that arises from a subgroup of $U_n$, for both the Bergman space and the Fock space. The density of Toeplitz operators in the quasi-radial Toeplitz algebra is known \cite{DO22, EM16}. We present a new result: Toeplitz operators are also dense in the quasi-radial Toeplitz algebra over the Bergman space.  

 \subsection{The Bergman space}
  Consider the unit ball in $\C^n$ with the normalized volume measure $dV$. Then the Bergman space, denoted $\Berg$, is the set of all holomorphic functions on the unit ball $\B^n$ in $\C^n$ that are square-integrable w.r.t. $dV$. Here we use the notation $\bergfock$ to denote both the Bergman space and the Fock space.

\subsection{\texorpdfstring{$k$}{k}-Quasi-radial operators}
 Let $k\in \N$ and suppose $\vn=\vn(k) = (n_1,\dots,n_k)\in \N^k$ s.t. $n=n_1+\cdots +n_k$. Consider the group $G=G_{\vn}:=U_{n_1}\times\cdots\times U_{n_k}$ where $\rU_{n_i}$ are $n_i\times n_i$ unitary matrices. Note that $G$ can be understood as a subgroup of $U_n$: the elements of $G$ can be identified with block diagonal matrices in $U_n$ and hence $G<U_n\ltimes \C^n$.
 Let $\bddqradX$ denote the algebra of all $\piG$-intertwining operators (quasi-radial operators). As a special case (when $k=1$), we get the widely studied radial radial operators  (see \cite{GV02}).
 And let $\toepalgrad$ denote the closed subalgebra of $\bddX$ generated by the set of all quasi-radial Toeplitz operators $\toep^{\text{Q-rad}}$.
 For $S \in \bddX$, the well-known quasi-radialization of $S$, denoted by $\QRad{S}$, is the convolution of $S$ with the constant function $1$:
$$\QRad{S}= \intG{\actm{A}S\actm{A}^*},$$
where $d\mu$ is the normalized Haar measure on $G$. Then $\QRad{S}$ is a bounded operator with
$$\|\QRad{S}\|\leq \|S\|.$$
It is well known that $\QRad{S}$ is in $\bddqradX$ for any $S\in \bddX$ and $\QRad{S}=S$ iff $S\in \bddqradX$ (see \cite{GV14}). In other words, quasi-radialization is a continuous projection from $\bddX$ to the algebra of quasi-radial operators $\bddqradX$. For completeness, we provide a proof of this below.

\begin{lemma}
    The map $S\mapsto \QRad{S}$ is a continuous projection from $\bddX$ to the algebra of quasi-radial operators $\bergfock$.
\end{lemma}

\begin{proof}
    Notice that for $g_1\in G$ and $f_1,f_2\in\bergfock$
     \begin{align*}
         \ip{\actop{g}(\QRad{S})f_1}{f_2}&=\int_G \ip{\piG(g_1g)S\piG(g_1g)^*f_1}{f_2} \haar{g}\\
         &=\int_G \ip{\piG(g)S\piG(g)^*\piG(g_1)f_1}{f_2} \haar{g}\\
         &=\ip{\QRad{S}f_1}{f_2}
     \end{align*}
     and hence $\QRad{S}\in\Hom$. Also if $S\in \Hom$, $\QRad{S}=S$.
     Hence for $S\in \bddX$,
     $$\QRad{\QRad{S}}=\QRad{S}.$$    
\end{proof}

Also the set of all Toeplitz operators is invariant under quasi-radialization and for $a\in L^\infty(\X^n)$,
    $$\QRad{T_a}=T_{\qrad{a}}$$
    where the bounded $G$-invariant function $\qrad{a}$ is given by
    $$\qrad{a}(z)=\int_G a(A^{-1}z) \haar{A};\ \ z\in\B^n.$$
    The proof of this fact is similar to Lemma \ref{lem:Toeplitz_invariance}.

\subsection{Quasi-radial Toeplitz algebra over the Fock space}
As an application of Theorem \ref{theo:density_subToeplitzalgebras} we get the already known fact that the radial Toeplitz operators are dense in $\toepalgrad$.

\subsection{Quasi-radial Toeplitz algebra over the Bergman space}
     
    In \cite{X15}, Xia proved that Toeplitz operators are dense in the Toeplitz algebra over the Bergman space. In his theory, he used the notion of weakly localized operators introduced in \cite{IMW13} in comparison with Fulshe's correspondence theory for the Fock space. Here we observe that the density of Toeplitz operators holds for quasi-radial Toeplitz algebras as well. 

    \begin{theorem}
     The set of all $k$-quasi radial Toeplitz operators $\toep^{\text{Q-rad}}$ on the Bergman space $\Berg$ is dense in the quasi-radial Toeplitz algebra $\toepalgrad$. 
     $$\toepalgrad= \toepalg\cap \bddqradBerg= \overline{\{T_a\mid a\in \bddqradBerg\}}$$
 \end{theorem}

 \begin{proof}
     The following inclusions are immediate:
     $$\overline{\{T_a\mid a\in \bddqradBerg\}}\subset \toepalgrad \subset (\toepalg \cap \bddqradBerg).$$
     To prove equality, let $S\in \toepalg \cap \bddqradBerg$ and let $\epsilon>0$. Then since Toeplitz operators are dense in $\toepalg$, we have $a\in L^\infty(\B^n)$ s.t.
     $$\|S-T_a\|<\epsilon.$$
     As $S\in \bddqradBerg$, we have $\QRad{S}=S$. Also $\QRad{T_a}=T_{\qrad{a}}$. Hence
     \begin{align*}
         \|S-T_{\qrad{a}}\|&=\|\QRad{S}-\QRad{T_a}\|\\
         &\leq \|S-T_a\|\\
         &<\epsilon.
     \end{align*}
     Then since $\qrad{a}$ is $G$-invariant, we have $S\in \overline{\{T_a\mid a\in \bddqradBerg\}}$. Therefore
     $$\toepalg \cap \bddqradBerg\subset \overline{\{T_a\mid a\in \bddqradBerg\}}$$
     completing the proof.
 \end{proof}


\section{SOT Density of Toeplitz operators in \texorpdfstring{$\bdd$}{bounded operators}}

In  \cite{E91} Engli\u{s} proved that the Toeplitz operators are strongly dense in $\bdd$. However, his proof is not constructive. 
We conclude our discussion by giving a constructive proof of this fact as an interesting application of the techniques in this note.

\begin{proposition}\label{prop:L1approximate_identity}
Let  $\{\psi_t\}$ be an approximate identity in $L^1(\C^n)$ and $S\in\bdd$. Then 
    $$\psi_t\ast S\to S\ \ \text{ in SOT}.$$
    Moreover, the algebra $\bdd$ is the closure of $\cbuop$ in strong operator topology.
\end{proposition}

\begin{proof}
    Let $f\in \F$ and let $\epsilon>0$. Then since $z\mapsto \actop{z}Sf$ is continuous, there exists a $\delta>0$ s.t. when  $|z|<\delta$,
    $$\|(\actop{z}S-S)f\|<\frac{\epsilon}{2}.$$
    Also choose $r>0$ s.t.  
    $$\|\psi_t\chi_{\{z\ |\ |z|>r\}}\|_1<\frac{\epsilon}{2(\|S\|+1)\|f_1\|}.$$
    Notice that \begin{align*}
        \ip{(\psi_t\ast S-S)f_1}{f_2}&= \ip{(\psi_t \ast S)f_1}{f_2}-\ip{Sf_1}{f_2}\\
        &= \int \ip{L_zSf_1}{f_2} \ \psi_t(z)\ dz- \ip{Sf_1}{f_2}\\
        &= \int \ip{(L_zS-S)f_1}{f_2} \ \psi_t(z) \ dz.
    \end{align*}
    Hence for all $t<r$,
    \begin{align*}
        |\ip{(\psi_t\ast S-S)f_1}{f_2}|&\leq \int_{\{|z|<r\}} \|(L_zS-S)f_1\|\|f_2\| \ \psi_t (z)\ dz\\
        &\hspace{3cm}+\int_{\{|z|>r\}} \|L_zS-S\|\|f_1\|\|f_2\| \ \psi_t(z)\ dz\\
        &\leq \Big( \frac{\epsilon}{2}+2\|S\|\|f_1\|\int_{\{|z|>r\}}  |\psi_t(z)|\haar{z} \Big) \|f_2\|\\
        &\leq ( \frac{\epsilon}{2}+ 2\|S\|\|f_1\| ) \|f_2\|\\
        &\leq \epsilon \|f_2\|.
    \end{align*}
    Hence $\|(\psi_t\ast S-S)f\|<\epsilon$ for all $t<r$ and $\psi_t\ast S\to S$ in SOT.
    The second statement is true because $\psi_t\ast S\in \cbuop$.
\end{proof}

\begin{theorem}
    For any bounded operator $S\in \bdd$ there exists a sequence of Toeplitz operators with uniformly continuous symbols which converges to $S$ is SOT. \end{theorem}

\begin{proof}
     Let $\{\psi_t\}$  be an approximate identity in $L^1(\C^n)$ s.t. each $\hat{\psi}_t$ has compact support. The convergence of $\psi_t\ast S$ to $S$ in SOT follows from Proposition \ref{prop:L1approximate_identity}. Since the Gaussian $\varphi$ is regular, there is $h_t\in L^1(\C^n)$ s.t. $\psi_t=\varphi \ast h_t$ by Wiener's division lemma \ref{lem:division}. Hence
    $$\psi_t \ast  S=h_t \ast (\varphi \ast S)=h_t \ast (T_{B(S)})=T_{h_t\ast B(S)}.$$ Finally, observe that $h_t\ast B(S)$ are all uniformly continuous as convolutions of $L^1$ and $L^\infty$ functions. 
\end{proof}

\noindent \textbf{Acknowledgements:} We would like to thank Gestur \'Olafsson for useful discussions.


\end{document}